\documentclass[12pt,a4paper]{article}
\usepackage[english]{babel}

\usepackage{hyperref}
\usepackage{pstricks}
\usepackage{pst-node}
\usepackage{amsmath,boxedminipage}
\usepackage{amssymb}
\usepackage{graphicx}
\usepackage{enumerate}
\usepackage{color}
\usepackage[text={15cm,24cm}]{geometry}
\usepackage[normalem]{ulem}
\usepackage[ansinew]{inputenc}
\usepackage{ulem}
\usepackage{float}
\usepackage{subcaption}
 \usepackage{algorithm,algorithmic}

\usepackage{color}

\newenvironment{proof}{{\bf Proof}:\ }%
   {~\ \hfill $\Box$\vspace{0,5cm}}

    {~\ \hfill$\Box$\vspace{0,5cm}}

\newtheorem{theorem}{Theorem}[section]

\newtheorem{conj}[theorem]{Conjecture}
\newtheorem{coro}[theorem]{Corollary}

\graphicspath{{.}{graphics/}}
\newrgbcolor{lightlightlightgray}{0.9 0.9 0.9}

\numberwithin{equation}{section}

\begin{document}
\title{Antimagicness  of join graphs}

\author{Gr\'egoire Beaudoire\footnotemark[1] \and C\'edric Bentz\footnotemark[1] \and
Christophe\ Picouleau\footnotemark[1] \footnotemark[2]
}
\date{\today}

\def\thefootnote{\fnsymbol{footnote}}

\footnotetext[1]{ \noindent
Conservatoire National des Arts et M\'etiers, CEDRIC laboratory, Paris (France). Email: {\tt
gregoire.beaudoire@lecnam.net,cedric.bentz@cnam.fr,christophe.picouleau@cnam.fr}}

\footnotetext[2]{ \noindent Corresponding author}

\maketitle

\begin{abstract}
An antimagic labelling of a graph $G = (V,E)$ is a bijection from $E$ to $\{1,2, \ldots, |E|\}$, such that all vertex-sums are pairwise distinct, where the vertex-sum of each vertex is the sum of labels over edges incident to this vertex. A graph is antimagic if it has an antimagic labelling.
Hartsfield and Ringel in 1990  stated the celebrated  conjecture: Every connected graph other than $K_2$ is antimagic. We prove the conjecture  for join graphs with at least three vertices.

 \vspace{0.2cm}
\end{abstract}

\section{Introduction}\label{intro}
We only consider finite, simple and undirected graphs $G = (V,E)$ with $|V| = n$ and $|E| = m$. We denote by $d(v)$ the degree of a vertex $v \in 
V$ and by $\Delta(G) = \max\limits_{v \in V} d(v)$. By $K_n$ and $K_{p,q}$ we denote respectively the complete graph on $n$ vertices and  the complete bipartite graph on $p$ and $q$ vertices. 

Given a graph $G = (V,E)$, let $f : E \rightarrow \{1,2,\ldots, m\}$ be a bijective labelling of the edges of $G$. For each vertex $u \in V$, we will write $\sigma(u) = \sum\limits_{v \in V | uv \in E} f(uv)$ the sum of labels over edges incident to $u$. If all the values of $\sigma(u)$ are pairwise distinct, then $f$ is called an \emph{antimagic labelling} of $G$. If $G$ admits at least one antimagic labelling, $G$ is said to be \emph{antimagic}.

Antimagic labelling was originally introduced by Hartsfield and Ringel in 1990 \cite{hartsfield1990}, with the following conjecture:

\begin{conj}
    Every connected graph other than $K_2$ is antimagic.
\end{conj}

Despite numerous results on this subject the conjecture is still open even for the class of trees. 

The topic is the focus of Chapter 6 in the dynamic survey, updated yearly, on graph labelling by J. Gallian \cite{survey}.\\

In this note we show that if $G$ is a join graph with at least three vertices then $G$ is antimagic.

 Before presenting our result  we first give some  terminology and notations.
 Let $G_1$ and~$G_2$ be two vertex disjoint graphs. The {\it join} operation $\otimes$ adds an edge between every vertex of $G_1$ and every vertex of $G_2$. The {\it union} operation $\oplus$ creates the disjoint union $G_1 \oplus G_2$ of $G_1$ and $G_2$, which is the graph with vertex set $V(G_1)\cup V(G_2)$ and edge set $E(G_1)\cup E(G_2)$. 

A graph $G=(V,E)$ is a {\it join graph} if it is the join of two vertex disjoint graphs, that is $G=G_1\otimes G_2$.

Among the class of join graphs a well-known subclass is the one of connected cographs \cite{BLS99}. We define this subclass here: a graph $G$ is a cograph if and only if $G$ can be generated from $K_1$ by a sequence of operations, where each operation is either $\otimes$ a join or $\oplus$ a union operation. 

Such a sequence corresponds to  a decomposition tree $T$, which has the following properties:
\begin{itemize}
\item [1.]  Its root $r$ corresponds to the graph $G_r=G$.
\item [2.]  Every leaf $x$ of $T$ corresponds to exactly one vertex of $G$, and vice versa, implying that $x$ corresponds to a unique single-vertex graph $G_x=K_1$.
\item [3.]  Every internal node $x$ of $T$ has at least two children, is either labelled $\oplus$ or $\otimes$, and corresponds to an induced subgraph $G_x$ of $G$ defined as follows:
\begin{itemize}
\item if $x$ is a $\oplus$-node, then $G_x$ is the disjoint union of all graphs $G_y$ where $y$ is a child of $x$;
\item if $x$ is a $\otimes$-node, then $G_x$ is the join of all graphs $G_y$ where $y$ is a child of $x$.
\end{itemize}
\end{itemize}
A cograph $G$ may have more than one such tree but has exactly one unique tree~\cite{CHS81}, called the {\it cotree} $T_G$ of $G$, if the following additional property is required: 
\begin{itemize} 
\item [4.] Labels of internal nodes on the (unique) path from any leaf to $r$ alternate between $\oplus$ and $\otimes$.
\end{itemize}

In \cite{DPPR18} Diner et al. show how to modify the cotree $T_G$ into a equivalent binary tree that will be used later into our proof:

 Whenever an internal node~$x$ of $T_G$ has more than two children $y_1$ and $y_2$, we remove the edges $xy_1$ and $xy_2$ and add a new vertex $x'$ with 
edges $xx'$, $x'y_1$ and $x'y_2$. If~$x$ is a $\oplus$-node, then $x'$ is a $\oplus$-node, and if $x$ is a $\otimes$-node, then $x'$ is a 
$\otimes$-node. Applying this rule exhaustively yields a tree in which each internal node has exactly two children. 

From a complexity point of view  the authors show that  this modification of $T_G$  takes linear time.

 Hence the connected cographs $G$ are such that $G=G_1\otimes G_2$ where $G_1,G_2$ are two cographs, so they are join graphs.
 
 Recall also that a graph $G$ is a cograph if and only it is $P_4$-free \cite{BLS99}.

 Among the join graphs, some well-known examples are: $K_n$ the complete graph on $n$ vertices, $K_{n,m}$ the complete bipartite graph with side sizes $n,m$, the complete $k$-partite graph $K_{n_1,n_2,\cdots,n_k}$, which is a connected cograph, $W_k=C_k\otimes K_1$ the wheel with $k$ spokes which is not a cograph when $k\ge 5$, and any graph $G$ when its complement $\bar G$ is not connected.

Some results concerning special cases of join graphs are given by Alon et al.  and by Ba\u{c}a et al.  (see \cite{alon2004},\cite{Baca15}).

 \section{The antimagicness of join graphs}
We prove here our main result.
 \begin{theorem}
    Join graphs $G=(V,E)$ with $\vert V\vert\ge 3$  are antimagic.
\end{theorem}
\begin{proof}
Let $G=G_1\otimes G_2$ be a join-graph with $\vert V(G_1)\vert=n_1,\vert V(G_2)\vert=n_2$ and $n_1+n_2\ge 3$. Let $m_1,m_2$ be respectively the number of edges in $G_1,G_2$.

 We know from \cite{alon2004} that if $G\ne K_2$ is a graph  with $\Delta(G)\ge n-2$ then $G$ is antimagic. If $n_1\le 2$  or $n_2\le 2$  then $\Delta(G)\ge n-2$ and so $G$ is antimagic. 
  
From now on, we assume that $n_1,n_2\ge3$.

Let $f$ be the following labeling of $G$: in a first stage we arbitrarily label the edges of $G_1,G_2$ with $\{1,2,\ldots,m_1+m_2\}$. Then we order the sum of the labels in $G_1$ and $G_2$ such that $\sigma(v_1)\le \sigma(v_2)\le\cdots\le\sigma(v_{n_1})$ and $\sigma(w_1)\le \sigma(w_2)\le\cdots\le\sigma(w_{n_2})$, respectively. 

  In the second stage we label the edges between $V(G_1)$ and $V(G_2)$. Let $m'=m_1+m_2$. First, we label the edges incident to $v_1$ with $\{m'+1,m'+2,\ldots, m'+n_2\}$ in the order $v_1w_1,v_1w_2,\ldots v_1w_{n_2}$. Second, we label the edges incident to $v_2$ with $\{m'+n_2+1,m'+n_2+2,\ldots, m'+2n_2\}$ in the order $v_2w_1,v_2w_2,\ldots v_2w_{n_2}$. We proceed in the  same way for $v_3,\ldots ,v_{n_1}$.
  
  We have the two following properties. First, concerning $G_1$, we have:
  
  \begin{equation}\sigma(v_{i+1})-\sigma(v_i)\ge n_2^2\ge 9, \forall~1\le i\le n_1-1\label{Equation_Sigma_Des_vi}\end{equation}
  
  Second, concerning $G_2$, we have:
  
  \begin{equation}\sigma(w_{j+1})-\sigma(w_j)\ge  n_1\ge 3, \forall~1\le j\le n_2-1\label{Equation_Sigma_Des_wj}\end{equation} 

We suppose that there are $v_i \in V(G_1)$ and $w_j \in V(G_2)$ such that $\sigma(v_i) =\sigma(w_j)$, otherwise we already have an antimagic labelling. In such a case, we say that $v_i$ and $w_j$ are in conflict. We will process iteratively until such a pair 
$(i,j)$ does not exist anymore for some $j < n_2$; any conflict involving $w_{n_2}$ will be solved afterwards. At each step of the process, for one or two vertices $z\in V(G_1)\cup V(G_2)$, the labels of two edges incident  to $z$ are swapped. Any such swap will be guaranteed not to affect the sorting of the $v_i$'s and $w_j$'s by increasing order of their value of $\sigma$.

For any $1 \leq s \leq n_1$ and $1 \leq t \leq n_2$, we define the sets $V_s = \{v_s, v_{s+1}, \ldots, v_{n_1}\}$ and $W_t = \{w_t, w_{t+1}, \ldots, w_{n_2}\}$.

We define two sets of vertices $V'$ and $W'$, initially $V' = V_1 = V(G_1)$ and $W' = W_1 = V(G_2)$. The sets $V'$ and $W'$ will be equal to some $V_s$ and $W_t$ during the process, with increasing values of $s$ and $t$, and will verify the following properties at any point of each iteration of the process:

\begin{enumerate}[(i)]

    \item Any conflict remaining in the graph is between a vertex in $V'$ and a vertex in $W'$;
    
    \item Any vertex in $V' \cup W'$ has not been involved in a label swap yet, meaning all its incident edges did not have their labels changed;

    \item Any vertex in $(V(G_1) \setminus V') \cup (V(G_2) \setminus W')$ will not be involved in any label swap during the rest of the process.

\end{enumerate}


We describe an iteration: we choose the pair $(i,j)$ such that $i$ is minimum, meaning that $(i,j)$ is the smallest such pair, since the vertices in $V(G_1)$ and $V(G_2)$ are numbered by increasing order of their values of $\sigma$.
Let us denote $f(v_iw_j) = p$.  

Firstly, we suppose that $ j < n_2$. We have $V' = V_p$ and $W' = W_q$, for some $p$ and $q$. Note that necessarily, since $v_i$ and $w_j$ are in conflict, $i \geq p$ and $j \geq q$, and any vertices $v_{i'}$ and $w_{j'}$, with $i' > i$ and $j' > j$, are in $V'$ and $W'$.
  We consider the two following cases:
\begin{itemize}
\item  $\sigma(w_{j+1})-1\ne \sigma(v_k), \forall k\ge i+1$:  We have $f(v_iw_{j+1}) = p + 1$, as $w_j,w_{j+1} \in W'$. Swapping the two labels $p$ and $p + 1$ we increase $\sigma(w_j)$ by one and decrease $\sigma(w_{j+1})$ by one, all the other values of $\sigma$ stay unchanged. Since $\sigma(w_{j+1})-\sigma(w_j) \ge 3$ from Property \ref{Equation_Sigma_Des_wj}, we still have $\sigma(w_j) <\sigma(w_{j+1})$, and, from Property \ref{Equation_Sigma_Des_vi}, $w_j$ cannot be in conflict with any $v_h$, for $h=1,\dots, n_1$, after the modification.

Since $(i,j)$ was the smallest pair such that $v_i$ and $w_j$ were in conflict, we know that any $v_{i'}$, with $i' < i$, and any $w_{j'}$, with $j' < j$, cannot be in conflict with another vertex. Since we also solved the conflict between $v_i$ and $w_j$, and we know that by assumption $w_{j+1}$ is not in conflict with any other vertex either, we can update our sets $V'$ and $W'$: $V' = V_{i+1}$ and $W' = W_{j+2}$. It is then easy to see that $V'$ and $W'$ still satisfy (i), (ii) and (iii) after their update.

\item $\sigma(w_{j+1})-1= \sigma(v_k)$ for some $k\ge i+1$: we have $f(v_iw_{j+1}) = p + 1$, as $w_j,w_{j+1} \in W'$, and since $\sigma(v_k)-\sigma(v_i)\ge 9$ from Property \ref{Equation_Sigma_Des_vi} we have $\sigma(w_{j+1})-\sigma(w_j)\ge 10$. Let  $f(v_kw_j) = q,f(v_kw_{j+1}) = q + 1$. Swapping the  labels $p$ and $p + 1$ and the labels $q$ and $q+1$ we increase $\sigma(w_j)$ by two and decrease $\sigma(w_{j+1})$ by two, all the other values of $\sigma$ stay unchanged. Since $\sigma(w_{j+1})-\sigma(w_j) \ge 10$ we still have $\sigma(w_j) <\sigma(w_{j+1})$, and, from Property \ref{Equation_Sigma_Des_vi}, $w_j$ and $w_{j+1}$ cannot be in conflict with any $v_h$, for $h=1,\dots, n_1$, after the modification. From Properties \ref{Equation_Sigma_Des_vi} and \ref{Equation_Sigma_Des_wj}, $v_k$ is not in conflict either.

As in the previous case, any $v_{i'}$, with $i' < i$, and $w_{j'}$, with $j' < j$, cannot be in conflict with another vertex. Note that, since we had a conflict between $v_i$ and $w_j$, and we had $\sigma(v_k) = \sigma(w_{j+1}) - 1$, it is impossible for any $v_{i'}$, with $i < i' < k$, to be in conflict with another vertex (as from Property \ref{Equation_Sigma_Des_vi} $\sigma(v_i) < \sigma(v_{i'}) < \sigma(v_k)$ and $\sigma(w_j) < \sigma(v_{i'}) < \sigma(w_{j+1})$). We can then update our sets $V'$ and $W'$: $V' = V_{k+1}$ and $W' = W_{j+2}$. Again, it is then easy to see that $V'$ and $W'$ still satisfy (i), (ii) and (iii) after their update.
\end{itemize}

Now we consider the case  where $j=n_2$. There is at most one conflict left in the graph, between $w_{n_2}$ and some $v_i$. From Properties (i) and (ii) verified by $V'$ and $W'$, we have necessarily that $v_i \in V'$ and that the labels of the edges incident to $v_i$ have not been changed during the process. We have the two following cases:

\begin{itemize}

\item  $i<n_1$: We have $f(v_{i+1}w_{n_2}) = p + n_2$, as $w_{n_2} \in W'$. Swapping the two labels $p$ and $p + n_2$ we increase $\sigma(v_i)$ by $n_2$ and decrease $\sigma(v_{i+1})$ by $n_2$, the other values of  $\sigma$ stay unchanged. Since from Property \ref{Equation_Sigma_Des_vi} $\sigma(v_{i+1})-\sigma(v_i) \ge n_2^2$ and $n_2\ge 3$ (meaning $n_2^2 > 2n_2$), we still have $\sigma(v_{i'}) <\sigma(v_{i'+1})$ for each $i' \in \{i, \dots, n_1\}$. Since $\sigma(v_i)$ strictly increases, we have that after the modification $\sigma(v_{i'}) > \sigma(w_h)$ for each $i' \in \{i, \dots, n_1\}$ and each $h \in \{1,\dots,n_2\}$, and since $j=n_2$ no more conflict can arise, and the labelling is antimagic;

\item $i=n_1$: We have  $f(v_{n_1}w_{n_2})=m$ and $f(v_{n_1}w_{n_2-1})=m-1$, as $v_{n_1} \in V'$. If there is no $v_k$ such that $\sigma(w_{n_2-1})+1=\sigma(v_k)$ then swapping the two labels $m$ and $m- 1$ we increase $\sigma(w_{n_2-1})$ by one and decrease $\sigma(w_{n_2})$ by one, the other values of $\sigma$ stay unchanged. Since after the previous modifications we have $\sigma(w_{n_2})-\sigma(w_{n_2-1})\ge 3$ now we still have $\sigma(w_{n_2})>\sigma(w_{n_2-1})$. Moreover, from Property \ref{Equation_Sigma_Des_vi}, $w_{n_2-1}$ and $w_{n_2}$ cannot be in conflict with any $v_h$, for $h=1,\dots, n_1$, and hence the labelling is antimagic.

Now, we suppose that there exists $v_k$ such that $\sigma(w_{n_2-1})+1=\sigma(v_k)$. Let $f(v_kw_{n_2})=q$. If $v_k$ has not be involved in a label swap before then we have $f(v_kw_{n_2-1})=q-1$, so swapping the  labels $m$ and $m - 1$ and the labels $q$ and $q-1$ we increase $\sigma(w_{n_2-1})$ by two and decrease $\sigma(w_{n_2})$ by two, the other values of $\sigma$ stay unchanged. Since the previous swaps did not modify the values of $\sigma(v),v\in V(G_1)$, we have  $\sigma(v_{n_1})-\sigma(v_k)\ge 9$ (thanks to Property \ref{Equation_Sigma_Des_vi}) and thus $\sigma(w_{n_2})-\sigma(w_{n_2-1})\ge 10$, so we still have $\sigma(w_{n_2})>\sigma(w_{n_2-1})$ after the modification. Again, from Property \ref{Equation_Sigma_Des_vi}, $w_{n_2-1}$ and $w_{n_2}$ cannot be in conflict with any $v_h$, for $h=1,\dots, n_1$, and hence the labelling is antimagic.

 Lastly, we suppose that $v_k$  has been involved in a label swap before and $f(v_kw_{n_2-1})\ne q-1$, otherwise we process as above. Since $v_k$ has been involved in exactly one swap, due to Properties (ii) and (iii), either $f(v_kw_{n_2-1})=q+1$ (when $f(v_kw_{n_2-1})$ and $f(v_kw_{n_2})$ were swapped, i.e., $v_k$ was in conflict with $w_{n_2-1}$) or $f(v_kw_{n_2-1})=q-2$ and $f(v_kw_{n_2-2})=q-1$ (when $f(v_kw_{n_2-1})$ and $f(v_kw_{n_2-2})$ were swapped, i.e., $v_k$ was in conflict with $w_{n_2-2}$). The former case is impossible since after the swap $\sigma(w_{n_2-1})$ is greater than $\sigma(v_k)$, and hence it is not possible to have $\sigma(w_{n_2-1}) + 1 = \sigma(v_k)$. This means that we have $f(v_kw_{n_2-1})=q-2$. Hence, swapping the labels $m$ and $m- 1$ and the labels $q$ and $q-2$, we increase $\sigma(w_{n_2-1})$ by three and decrease $\sigma(w_{n_2})$ by three, the other values of $\sigma$ stay unchanged. Since, as before, $\sigma(w_{n_2})-\sigma(w_{n_2-1})\ge 10$, we still have $\sigma(w_{n_2})>\sigma(w_{n_2-1})$, and, from Property \ref{Equation_Sigma_Des_vi}, $w_{n_2-1}$ and $w_{n_2}$ cannot be in conflict with any $v_h$, for $h=1,\dots, n_1$; hence, the labelling is antimagic.

\end{itemize}
Thus in any case we obtain an antimagic labeling of $G=G_1\otimes G_2$ .
\end{proof}

We immediately obtain the following corollary, since all the graphs that are listed below are join graphs.
 \begin{coro} For a graph $G$ with at least three vertices, if $G=K_n$, or if
$G=K_{n,m}$, or if $G=K_{n_1,n_2,\cdots,n_k}$, or if $G=W_k=C_k\otimes K_1$, or if $G$ is a connected cograph, or if the complement $\bar G$ of $G$ is not connected, then $G$ is antimagic.
 \end{coro}
\section{Conclusion}
We have proved that join graphs distinct from $K_2$ are antimagic. Among them an important subclass is the one of connected cographs. A complete bipartite graph $K_{p,q}$ being a connected cograph, our result gives an alternative proof of the one given by Alon et al. \cite{alon2004}. Concerning cographs that are not connected, and do not have $K_2$ as a component, it is easy to prove that, if $G=K_{i_1}\oplus K_{i_2}\oplus K_{i_p},3\le i_1\le i_2\le\cdots\le i_p,$, i.e., $G$ is the union of $p$ cliques, then $G$ is antimagic. Furthermore Chen and Tay in \cite{chen25} prove the following result: if $G=mP_3,m\ge 1,$ that is $G$ consists of $m$ disjoint induced paths on three vertices, i.e., $m$ disjoint complete bipartite $K_{1,2}$, then $G$ is antimagic if and only if $m\ne 2$; they also give conditions on $m,t\ge 1$ for the antimagicness of $G=mC_3\oplus tP_3$. A future work will be to complete these results and characterize the disconnected cographs that are antimagic, even for the case where $G$ is a collection of disjoint complete bipartite graphs.
Since the class of cographs is the same as the class of $P_4$-free graphs, another open problem consists in proving the antimagicness for connected $P_k$-free graphs for $k\ge 5$. 

As the complement of a disconnected graph is a join graph, we have proved that such a graph is antimagic. Another open problem consists in studying the disjoint unions of join graphs (that contain the disconnected cographs).

\end{document}